\documentclass[12pt]{amsart}
\usepackage{amsmath,amsfonts,amsthm,amsopn}
\usepackage{graphicx}
\usepackage{hyperref}
\usepackage[usenames]{color}

\theoremstyle{plain}
\newtheorem{theorem}{Theorem}
\newtheorem{proposition}{Proposition}

\newtheorem{Cor}{Corollary}
\theoremstyle{definition}

\theoremstyle{remark}
\newtheorem{remark}{Remark}

\begin{document}
\title{An  Inverse Problem for  Localization Operators}
\author{Lu\'{\i}s Daniel Abreu }
\address{Institut f\"ur Mathematik, Universit\"at Wien, Alserbachstrasse 23
  A-1090 Wien,  on leave from Department of Mathematics of University of Coimbra, Portugal.
 }\email{daniel@mat.uc.pt}
\author{Monika D\"{o}rfler.}
\address{Institut f\"ur Mathematik, Universit\"at Wien, Alserbachstrasse 23.
 }
 \email{monika.doerfler@univie.ac.at} 
\thanks{Daniel Abreu was supported by the ESF activity "Harmonic and Complex Analysis and its Applications", by FCT (Portugal) project PTDC/MAT/114394/2009 and CMUC through COMPETE/FEDER. 
Monika D\"orfler was supported by the Austrian Science Fund (FWF):[T384-N13] {\em Locatif}}

\begin{abstract}
A classical result of time-frequency analysis, obtained by I.~Daubechies in 1988, states that  the  eigenfunctions  of a time-frequency localization operator with circular localization domain and  Gaussian 	analysis window are the Hermite functions.
In this contribution,  a converse of Daubechies' theorem is proved. More precisely, it is shown that, for simply connected localization domains, if one of the eigenfunctions of a time-frequency localization operator with Gaussian window is a Hermite function, then its localization domain is a disc.
The general problem of obtaining, from some knowledge of its eigenfunctions,  information about the symbol of a time-frequency localization operator, is denoted as {\it the inverse problem}, and the problem studied by Daubechies  as {\it the direct problem} of time-frequency analysis.
Here, we  also solve the corresponding problem for wavelet localization, providing the inverse problem analogue of the direct problem studied by Daubechies and Paul.\\
\today
\end{abstract}

\maketitle

\section{Introduction}
\begin{center}\begin{it}
$[...]$ The fundamental things apply\\
As time goes by.\end{it}\\
\begin{scriptsize}(Herman Hupfeld)\\[.5cm]\end{scriptsize}\end{center}
\begin{sloppypar}

\emph{As time goes by}, most  real-life signals of interest  change their \emph{frequency} properties.
Therefore,  a signal description by means of  \emph{%
time-frequency analysis} is often preferable to   the signal's Fourier transform, which reliably yields frequency information, but without any localization in time.
The core purpose of time-frequency analysis is to represent a given signal
as a function in the time-frequency or in the time-scale plane. However, in
real world applications like optics and wireless communications, one can
only \textquotedblleft sense\textquotedblright\ a signal within a certain
region of those planes. This means that, in practice,  the part of
the signal outside the region of interest is neglected  and  only its \textquotedblleft
localized\textquotedblright\ version is observed. \emph{Localization operators} turn
this observation process into rigorous mathematical terms. They transform a
given signal into one that is localized in a  given region by reducing the
 signal energy outside that region to a negligible amount.

The first  approach to time-frequency  localization, introduced in 1961, consists in separately  selecting  time-
and frequency-content, and is  described in a famous
series of papers known as the \textquotedblleft Bell labs
papers\textquotedblright. We refer to Slepian's review \cite{sl83} for
an account of this beautiful body of work. In 1988,  Daubechies added a new  perspective by introducing operators, that localize directly
in the time-frequency plane~\cite{da88} and, together with Paul~\cite{dapa88}, 
extended the analysis to the time-scale plane. The time-frequency plane is
associated to the \emph{short-time Fourier transform} and the time-scale
plane is associated to the \emph{wavelet Transform}. We begin our
presentation by defining the short-time Fourier transform, which leads to the
concept of time-frequency localization operators.

The short-time Fourier (or Gabor) transform of a function or distribution $f$ with respect to a
window function $g\in L^{2}(\mathbb{R})$ is defined to be, for $z=(x,\xi
)\in \mathbb{R}^{2}$: 

\begin{equation}
\mathcal{V}_{g}f(z)=\mathcal{V}_{g}f(x,\xi )=\int_{\mathbb{R}}f(t)\overline{%
g(t-x)}e^{-2\pi i\xi t}dt  \label{Gabor}
\end{equation}
where the overline denotes complex conjugation.
We let $\pi (z)g(t)  = g(t-x)e^{2\pi i\xi t}$ and observe that $f$ can be resynthesized from $\mathcal{V}_{g}f$ as 
\begin{equation*}
f=\frac{1}{\|g\|_2^2}\int_{\mathbb{R}^{2}}\mathcal{V}_{g }f(z)\pi (z)g
\,dz .
\end{equation*}%
Given a symbol $\sigma \in L^{1}(\mathbb{R}^{2})$, \textit{time-frequency localization operators} $H_{\sigma ,g}$  are defined by 
\begin{equation*}
H_{\sigma ,g}f=\int_{\mathbb{R}^{2}}\sigma (z)\mathcal{V}_{g }f(z)\pi (z)g
\,dz=\mathcal{V}_{g}^{\ast }\sigma \mathcal{V}_{g}f.
\end{equation*}%
In signal processing it is very common to modify a signal $f$ by acting on its time-frequency coefficients $\mathcal{V}_{g}f$, for example in order to achieve noise reduction~\cite{ma09-1}; the corresponding
localization operators have been the object of  research in time-frequency analysis,~\cite{feno01,cogr03}.
In~\cite{da88}, Daubechies considered the window $g(t)=\varphi (t)=2^{\frac{1%
}{4}}e^{-\pi t^{2}}$, the symbol $\sigma (z)=\chi _{\Omega }(z)$, i.e. the indicator function of  a set $\Omega\subset\mathbb{R}^2$, and
investigated the eigenvalue problem 
\begin{equation}
H_\Omega f :=H_{\chi _{\Omega ,\varphi }}f=\lambda f  \label{eq:EVP}
\end{equation}%
for the case where $\Omega  $ is a disc centered at zero. She concluded, that in this
situation, the eigenfunctions of $H_{\chi _{\Omega ,\varphi }}$ are the
Hermite functions. Consequently, since,  $H_{\Omega_2\setminus \Omega_1} = H_{\Omega_2}- H_{\Omega_1}$ for two sets $\Omega_1\subset\Omega_2$, the Hermite functions are also eigenfunctions with respect to domains in the form of an annulus centered at zero and for any union of annuli.

The problem (\ref{eq:EVP}) is important in time-frequency
analysis, because its solutions are the functions with best concentration in
the subregion $\Omega $ \ of the time-frequency plane, where we consider the time-frequency concentration of a function $f$ in $\Omega\subset\mathbb{R}^2$  defined as 
\begin{equation}\label{Eq:TFcon}
	\mathcal{C}_\Omega (f) = \frac{\int_\Omega |\mathcal{V}_\varphi f (z)|^2 dz}{\|f\|_2^2}\, .
\end{equation}


In this paper we will be concerned with the inverse situation of the one
considered by Daubechies. This leads us to the following question:

\begin{itemize}
\item Given a localization operator with unknown localization domain $\Omega 
$, can we recover the shape of $\Omega $\ from information about its
eigenfunctions?
\end{itemize}

This is a new type of inverse problem, and we will call it the \textquotedblleft 
\emph{inverse problem of time-frequency localization\textquotedblright }. We solve the problem in the case where explicit computations can be
made, which is the set-up of~\cite{da88}.
Our main contribution is the following.

\begin{theorem}
\label{Th:main1}
Let $\Omega\subset\mathbb{R}^2$ be simply connected. If one of the eigenfunctions of the localization operator $%
H_{\Omega }$ is a Hermite function, then $\Omega $ must be a disk centered at $0$. 
\end{theorem}
Let us briefly discuss some motivations for our studies and consequences of our result.
Hermite functions have been proposed as  modulating pulses 
in pulse-shape modulation for ultra-wideband (UWB) communication, mainly due to their maximal joint concentration  in  time and frequency, cf.~\cite{sica07,pabywazh03,ghmihako02} and references therein. The receiver of communication system often applies a filter with the modulation pulses as eigenfunctions corresponding to large eigenvalues, in order to suppress random noise that accumulates during transmission. In this situation,  Theorem~\ref{Th:main1} shows that the filter must be designed with a circular  localization domain. 
Furthermore, if the   filter on  the receiver side is known or designed to have  one single Hermite function as an eigenfunction, it is possible to  guarantee the location of the time-frequency  plane that the filter is sensing. This is particularly important in UWB communication, where  the permitted spectrum is officially prescribed, cp.~\cite{FCC02}.

This last remark also hints at an additional possible application, which is system identification. Identification of linear time-variant  systems is a notoriously difficult task in general, cp.~\cite{ka63,be69,pfwa06}. While a linear time-invariant system is straight-forwardly identified by sending an impulse to the system and retrieving the impulse response, \cite[Sec.~4.2]{gr01}, no similar method exists for general linear time-variant systems. By our result, and for a system that is known to be a localization operator of the form $H_\Omega$ for some time-frequency region $\Omega$, we may send any Hermite function to the system and judge from the response, whether $\Omega$ can be a disk. In the positive case, one can then evaluate the size of the disk by subsequently sending additional Hermite functions and evaluating the resulting scaling factors. Obviously, this approach should be extended to other shapes and its feasibility will be the topic of further research.

In analogy to Theorem~\ref{Th:main1}, we will also consider an inverse problem for wavelet localization
operators. Here, we show that the domain of localization of the localization
operators investigated by Daubechies and Paul~\cite{dapa88} is a
pseudohyperbolic disc in the upper half plane whenever one of the operator's
eigenfunctions is the Fourier transform of a Laguerre function. We will
essentially use methods from complex analysis and our techniques are
strongly influenced by the ideas contained in~\cite{an00} and~\cite{se91}.

This paper is organized as follows. Section~\ref{Sec:EF_Loc} collects some properties of the eigenfunctions of localization operators with respect to radially weighted measures and Section~\ref{Se:LoDoMo} deduces the geometry of localization domains under the assumption of  orthogonality of any  single monomial to almost all monomials. The corresponding inverse problem for Gabor localization is studied in  Section~\ref{Se:InvProbGab} and Section~\ref{SE:WavLoc} is devoted to the investigation of the inverse problem for wavelet localization.

\section{Double orthogonality and reproducing kernel Hilbert spaces}\label{Sec:DO_RPH}
This section is devoted to  the properties of complex monomials, namely their double orthogonality with respect to any  radially weighted measures and the consequences of this property. 
\subsection{Eigenfunctions of Localization Operators}\label{Sec:EF_Loc}

Let $D_{a}$ denote a disk of radius $a$, $0<a\leq \infty $,  $d\mu
(z)=\mu (\left\vert z\right\vert )dz$  a radially weighted measure and $dz$  Lebesgue measure on $\mathbb{C}$.\newline In the sequel,
we will denote by $\mathcal{H}_{a}=L^{2}(D_{a},d\mu (z))$ the Hilbert space
of analytic functions $F$ on $\mathbb{C}$, such that 
\begin{equation*}
\Vert F\Vert _{\mathcal{H}}=\int_{D_{a}}|F(z)|^{2}d\mu (z)
\end{equation*}%
is finite. 
In  Proposition~\ref{Prop1} we collect the most important facts about the
\textquotedblleft direct problem\textquotedblright\ studied in~\cite{da88} 
\cite{dapa88} when transfered to  the complex domain. This point of view is essentially
contained in~\cite{se91}, but we have observed that both problems can be
understood as special cases of a more general formulation with general
radial measures on complex domains. This viewpoint is later reflected in our
derivation of the results about the inverse problems.

\begin{proposition}
\label{Prop1}Consider all radial
measures on disks $D_{R}$ with radius $R$ in the complex
plane, i.e. the measures constituted by the weighted measure $d\mu (z)=\mu
(\left\vert z\right\vert )dz$, defined on $D_{R}$, whose weight $\mu
(\left\vert z\right\vert )$ depends only on $r=|z|$.  The following statements are true:

\begin{description}
\item[(a)] The monomials are orthogonal on any disk $D_{R}$ centered at zero with radius $R$
in the complex plane and with respect to all concentric measures.\\
Consequently, the monomials are also orthogonal on any annulus centered at zero.

\item[(b)]   Assume $%
0<c_{n,a}<\infty $ for all moments $c_{n,a}$ of $\mu (\left\vert
z\right\vert )dz$. Then, the normalized monomials $e_{n,a}= z^{n}/\sqrt{(c_{2n+1,a})}$ constitute an orthonormal basis for $\mathcal{H}_{a}$.

\item[(c)] If, in addition, $\sum_{n\geq 0}(c_{2n+1,a})^{-1}|z|^{2n}$ is
finite for all $z\in D_{a}$, then $\mathcal{H}_{a}$ is a reproducing kernel
Hilbert space with reproducing kernel 
\begin{equation*}
K(\overline{z},w)=\sum_{n\geq 0}(c_{2n+1,a})^{-1}\overline{z}^{n}w
^{n}.
\end{equation*}

\item[(d)] The functions $F(z)=e_{n,a}$ are eigenfunctions of the problem%
\begin{equation}
\int_{D_{R}}F(z)K(\overline{z},w)d\mu (z)=\lambda
F(w).  \label{eigenvalue}
\end{equation}
\end{description}
\end{proposition}

\begin{proof}
(a) Orthogonality can directly be seen by 
\begin{equation}
\int_{D_{R}}z^{n}\overline{z}^{m}d\mu
(z)=\int_{0}^{R}r^{(n+m+1)}\int_{0}^{2\pi }e^{i\left( n-m\right) \theta
}d\theta \mu (r)dr=c_{2n+1,R}\delta _{n,m},  \label{Eq:Orth_Monom}
\end{equation}%
with $c_{n,R}=2\pi \int_{0}^{R}r^{n}\mu (r)dr$. 

(b) Consider a domain $D_{a}$, $R<a\leq \infty $ such that $%
\lim_{r\rightarrow a}d\mu (r)=0$.\ Since the power series $\sum_{n\geq
0}a_{n}z^{n}$ of an analytic function $F$ on $\mathbb{C}$ converges
uniformly on every $D_{R}$, we may interchange integral and summation in the
following equations: suppose that $\langle F,e_{n,a}\rangle =0$ for all $%
n\in \mathbb{Z}$, then 
\begin{align*}
0=& \frac{1}{\sqrt{c_{2m+1,a}}}\lim_{R\rightarrow a}\int_{D_{R}}\sum_{n\geq
0}a_{n}z^{n}\overline{z}^{m}\mu (|z|)dz \\
=& \frac{1}{\sqrt{c_{2m+1,a}}}\lim_{R\rightarrow a}\sum_{n\geq
0}a_{n}\int_{D_{R}}z^{n}\overline{z}^{m}\mu (|z|)dz \\
=& \frac{1}{\sqrt{c_{2m+1,a}}}\lim_{R\rightarrow a}a_{m}c_{2m+1,R}
\end{align*}%
which implies $a_{m}=0$ for all $m$ and hence $F\equiv 0$, which proves
completeness of the functions $\{e_{n,a}\}$ in\ $\mathcal{H}_{a}$.\newline
(c) We need to show that point evaluations of $F\in \mathcal{H}_{a}$ are
bounded. Expanding $F$ in terms of $\{e_{n,a}\}$, we observe that 
\begin{equation*}
|F(z)|=|\sum_{n\geq 0}\langle F,e_{n,a}\rangle \frac{z^{n}}{\sqrt{c_{2n+1,a}}%
}|\leq \Vert F\Vert _{\mathcal{H}_{a}}\cdot (\sum_{n\geq 0}\frac{1}{%
c_{2n+1,a}}|z|^{2n})^{\frac{1}{2}}.
\end{equation*}%
Thus, by the assumption on the growth of the moments, $\mathcal{H}_{a}$ is a
reproducing kernel Hilbert space. \newline
(d) Write $U$ for the operator which multiplies a function $F\in \mathcal{H}$
by the characteristic function of the circle $D_{R}$ and $P$ for the
orthogonal projection onto $\mathcal{H}_{a}$, given by the reproducing
kernel. Since $P(\frac{z^{n}}{\sqrt{c_{2n+1,a}}})=\frac{z^{n}}{\sqrt{%
c_{2n+1,a}}}$, we note that 
\begin{equation*}
0=\int_{D_{R}}e_{n,a}\overline{e_{n,a}}\mu (\left\vert z\right\vert
)dz=\int_{D_{a}}e_{n,a}PU(\overline{e_{n,a}})d\mu (z).
\end{equation*}%
and completeness of $e_{n,a}$ implies that $PUe_{n,a}=e_{n,a}$. Denoting by $%
K(\overline{z},w)$ the reproducing kernel of $\mathcal{H}_{a}$, the
functions $F(z)=e_{n,a}$ are eigenfunctions of problem \eqref{eigenvalue}.
\end{proof}

Using appropriate unitary operators (the so-called Bargmann and Bergman
transform, to be defined later in this paper), the solution to the general
problem just described can be shown to be equivalent to the solution of the
\textquotedblleft direct\textquotedblright\ problems considered in~\cite%
{da88} and~\cite{dapa88}. Indeed, the $d\mu (z)=e^{-\pi |z|^{2}}dz$ case can
be translated to the Gabor localization problem studied by Daubechies and
the case $d\mu (z)=(1-|z|^{2})^{\alpha }dz$ to the wavelet localization
studied by Daubechies and Paul. Details will be given in Section~\ref{Se:InvProbGab} and Section~\ref{SE:WavLoc}.

\subsection{The localization domain of monomials}\label{Se:LoDoMo}

We now turn to the general problem, given by \eqref{eigenvalue}. The
following, central proposition states that orthogonality of any monomial to almost
all other monomials with respect to a bounded, simply connected  domain $\Omega\subset\mathbb{C}$ forces $\Omega$
to be a disk centered at zero. We also consider more general domains as described in Theorem~\ref{Th:main1}(b). Note that we identify $\mathbb{R}^2$ with $\mathbb{C}$ for the geometric description. The proof is based on an idea of Zalcman~\cite{za87} and is essentially similar to the proof given in~\cite{an00}, but in
a more general setting, namely generalizing from area measure to general concentric measures. To adapt the original argument, we rely on Proposition~\ref{Prop1}.

\begin{proposition}\label{Th:LocMon}
Let $d\mu (z)$ be a positive, concentric
measure on $D_a\subseteq\mathbb{C}$ and consider a simply connected set  $\Omega\subset D_a$. 
Assume,  for some $m$ and $k\geq 0$  that 
\begin{equation}
\int_{\Omega }\left\vert z\right\vert ^{2m}\overline{z}^{k}d\mu (z)=\lambda
\delta _{k,0}\text{.}  \label{eq:exp_zero}
\end{equation}%
Then $\Omega$  must be a disk centered at zero. 
\end{proposition}

\begin{proof}
Since 
\[
\frac{\overline{zw}}{\overline{z}-\overline{w}}=-\frac{\overline{z}}{1-\frac{%
\overline{z}}{\overline{w}}}=-\sum_{n=1}^{\infty }\frac{\overline{z}^{n}}{%
\overline{w}^{n-1}}\text{,}
\]%
we have for every $z\in \Omega $ and $w$ such that $\left\vert w\right\vert
>\sup \{\left\vert z\right\vert ;z\in \Omega \}$, the following expansion:%
\[
\left\vert z\right\vert ^{2m}\frac{\overline{zw}}{\overline{z}-\overline{w}}%
=-\left\vert z\right\vert ^{2m}\left( \overline{z}+\frac{\overline{z}^{2}}{%
\overline{w}}+\frac{\overline{z}^{3}}{\overline{w}^{2}}+...\right) .
\]%
Integrating term wise and using \eqref{eq:exp_zero} yields%
\begin{equation}
\int_{\Omega }\left\vert z\right\vert ^{2m}\frac{\overline{zw}}{\overline{z}-%
\overline{w}}d\mu (z)=0,
\end{equation}%
hence 
\begin{align}
\int_{\Omega }\left\vert z\right\vert ^{2m}\frac{\left\vert z\right\vert
^{2}-\overline{z}w}{\left\vert z-w\right\vert ^{2}}d\mu (\left\vert
z\right\vert )=& \int_{\Omega }\left\vert z\right\vert ^{2m}\frac{\overline{z%
}}{\overline{z}-\overline{w}}d\mu (z) \label{eq:zeroform1} \\
=& \frac{1}{\overline{w}}\int_{\Omega }\left\vert z\right\vert ^{2m}\frac{%
\overline{zw}}{\overline{z}-\overline{w}}d\mu (z)=0.\label{formula}
\end{align}%
The left expression in \eqref{eq:zeroform1} is continuous as a function of $%
\overline{w}$ since the integrand is locally integrable in $z$. Therefore, %
\eqref{formula} holds on $\overline{\Omega ^{c}\text{.}}$ \\
We next show that $0$ is  inside  $\Omega $. Begin by observing that, for $\left\vert
w\right\vert >\sup \{\left\vert z\right\vert ;z\in \Omega \}$, we can expand
and integrate term wise so that 
\begin{equation}
\int_{\Omega }\left\vert z\right\vert ^{2m}\frac{1}{\overline{z}-\overline{w}%
}d\mu (z)=\frac{1}{\overline{w}}\int_{\Omega }\left\vert z\right\vert ^{2m}%
\frac{\overline{w}}{\overline{z}-\overline{w}}d\mu (z)=\frac{1}{\overline{w}}%
\lambda .  \label{formula1}
\end{equation}%
Let $C>\sup \{\left\vert z\right\vert ;z\in \Omega \}$. We let  $d(w,\Omega )$ denote  for the Euclidean distance between $w$ and $%
\Omega $, i.e. $d(w,\Omega ) = \inf_{w'\in\Omega}|w-w'|$. Then the following
pointwise estimate in $\overline{w}\in \Omega ^{c}$ holds:%
\[
\int_{\Omega }\left\vert \left\vert z\right\vert ^{2m}\frac{1}{\overline{z}-%
\overline{w}}\right\vert d\mu (z)\leq \frac{C^{2m}}{d(w,\Omega )}.
\]%
 This allows to extend (\ref{formula1}) by analytic continuation to 
$\overline{\Omega ^{c}}$. \\
Suppose now that $0\in \overline{\Omega ^{c}\text{%
.}}$ Then we can find a sequence of points $\{w_{n}\}$ contained in $\Omega
^{c}$ such that $w_{n}\rightarrow 0$. This would give%
\[
\lim_{n\rightarrow \infty }\int_{\Omega }\left\vert z\right\vert ^{2m}\frac{1%
}{\overline{z}-\overline{w_{n}}}d\mu (z)=\lim_{n\rightarrow \infty }\frac{1}{%
\overline{w_{n}}}\lambda _{m}=\infty \text{.}
\]%
On the other hand, because of the continuity of the left expression in $%
\overline{w}$, 
\[
\lim_{n\rightarrow \infty }\int_{\Omega }\left\vert z\right\vert ^{2m}\frac{1%
}{\overline{z}-\overline{w_{n}}}d\mu (z)=\int_{\Omega }\left\vert
z\right\vert ^{2m}\frac{1}{\overline{z}}d\mu (z),
\]%
and the integral on the right is bounded for every $m\geq 0$, since we are
assuming that $0\notin \Omega .$ This is a contradiction and we must have $%
0\in \Omega $. \\
Finally, we can consider $D_{R}$, the largest disc centered at
zero and contained in $\Omega $. Using Proposition~\ref{Prop1}(a), we can repeat the steps leading to (\ref{formula1}) with $%
D_{R}$ instead of $\Omega $. Pick a point $w_{0}\in \partial D_{R}\cap
\partial \Omega $. Then
\[
\int_{\Omega \backslash D_{R}}\left\vert z\right\vert ^{2m}\frac{\left\vert
z\right\vert ^{2}-\operatorname{Re}\overline{z}w_{0}}{\left\vert z-w_{0}\right\vert
^{2}}d\mu (z)=0.
\]%
Since, for $z\in \Omega \backslash D_{R}$,  $\left\vert \operatorname{Re}\overline{z}w_{0}\right\vert \leq \left\vert
z\right\vert \left\vert w_{0}\right\vert \leq |z|^{2}$, the integrand is
positive on $\Omega \backslash D_{R}$. This forces $\Omega \backslash D_{R}$
to be of area measure zero, which implies $\Omega =D_{R}$.

\end{proof}
\begin{figure}[h]
	\centerline{\includegraphics[scale = .3]{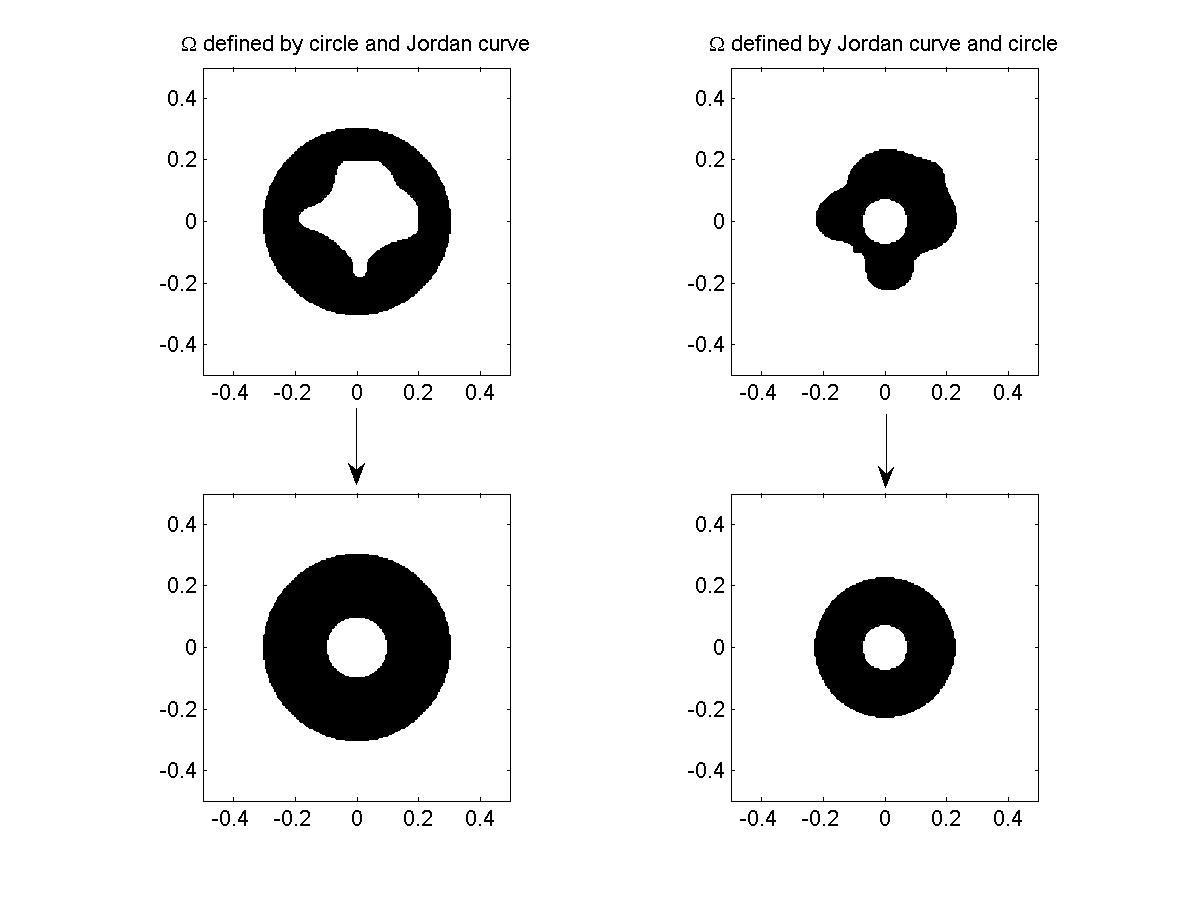}}
	\caption{The situation described in Corollary~\ref{Cor_JordanCirc}. $\Omega$ is an annulus. }
	\label{Fig2}
\end{figure}
For the next statement, we consider a more general situation. Let  $\gamma_j$, $j = 1,\ldots , n$ be a family of non-intersecting Jordan curves with interiors $I^{\gamma_j}$ such that $I^{\gamma_{j-1}}\subset I^{\gamma_j}$ for all $j>1$. \\
If $n$ is even, set $K = \frac{n}{2}$ and  let $\Omega_k = I^{\gamma_{2k}}\setminus I^{\gamma_{2k-1}}$ for $k = 1,\ldots, K$.\\  
If $n$ is odd, set $K = \frac{n+1}{2}$ and let $\Omega_1 = I^{\gamma_1}$ and $\Omega_k = I^{\gamma_{2k-1}}\setminus I^{\gamma_{2k-2}}$ for $k = 2,\ldots, K$.  For the situation just described, we set $\Omega =\bigcup_{k=1}^{K} \Omega_k$ and consider the corresponding localization operator. The next corollary shows that under the double orthogonality condition \eqref{eq:exp_zero}, all curves must contain $0$ in their interior. Furthermore, for $n=2$, if one of the two curves is a circle, $\Omega$ must be a annulus.
\begin{Cor}\label{Cor_JordanCirc}
\begin{itemize}
\item[(a)]
Let \eqref{eq:exp_zero} hold for  $\Omega =\bigcup_{k=1}^{K} \Omega_k$  defined by a family of nested Jordan curves as described above. Then all curves $\gamma_j$ must contain zero. 
\item[(b)] If $n = 2$ and $\gamma_j$ is a circle centered at $0$ for $j = 1$ or $j = 2$, then $\Omega$ is an annulus, see Figure~\ref{Fig2}. 
\end{itemize}
\end{Cor}
\begin{proof}
(a) We will show by induction, that $0$ must be inside all curves $\gamma_j$, $ j = 1, \ldots, n$.

\emph{Case }$n=1$. Then $\Omega $ is the interior of $\gamma_1$, therefore  simply connected,  and it follows from the proof of Proposition~\ref{Th:LocMon},  that $0\in \Omega$. 

\emph{Case }$\emph{n=2}$\emph{. }Then\emph{\ }$\Omega  = I^{\gamma_2}\setminus I^{\gamma_1}$ and  
 $I^{\gamma_1}$ is
simply connected.
 We  apply, by assuming that $0\in \overline{(I^{\gamma_2})^{c}\text{%
.}}$ the
argument used in the first paragraph of Case $n=1$ to show that $0\in \left(
\Omega \cup I^{\gamma_1}\right) $. 
Then, either $0\in \Omega $ or $0\in
I^{\gamma_1}.$ In the first case we  consider again $D_{R}$, the
largest disc centered at zero contained in $\Omega $ and argue as in \emph{Case }$\emph{n=1}$  to show that $\Omega =D_{R}$, which contradicts the assumption that $n = 2$. Therefore,   $0\in  I^{\gamma_1}$

\emph{Arbitrary }$n\in \mathbb{N}$\emph{. } Assume that, for $n-1$ curves, $0$ is inside all curves. For $n$ curves, we first show that  $0\in I^{\gamma_n}$, assume that $0\in\Omega_K$ and use, as before, the argument from  \emph{Case }$\emph{n=1}$ to show that this leads to $n = 1$. Consequently, $0$ must be inside the remaining $n-1$ curves and, by induction hypothesis, inside all curves $\gamma_j, j=1\ldots n$.\\

(b) First assume that $\Omega$ is a disk, centered at zero,  with a hole, in other words, that $\gamma_2$ is a circle. Then, $I^{\gamma_2}$ is a disk centered and zero, such that \eqref{eq:exp_zero} holds for $I^{\gamma_2}$ an therefore also for $I^{\gamma_1}$. Since the latter is simply connected, it must be a disk centered at $0$.\\
Now let $I^{\gamma_1}$ enclose a disk centered at $0$. We then consider the largest annulus $\Pi$ contained in $\Omega$, it is given by $\Pi = D_R\setminus I^{\gamma_1}$ where $D_R$ is the largest disk centered at zero and contained in $I^{\gamma_2}$. Due to Proposition~\ref{Prop1}(a), condition \eqref{eq:exp_zero} holds on $\Pi$ and we obtain (\ref{formula1}) with $\Pi$ instead of $\Omega $. Pick a point $w_{0}\in \partial D_{R}\cap \gamma_2 $. Then
\[
\int_{\Omega \backslash \Pi}\left\vert z\right\vert ^{2m}\frac{\left\vert
z\right\vert ^{2}-\operatorname{Re}\overline{z}w_{0}}{\left\vert z-w_{0}\right\vert
^{2}}d\mu (z)=0.
\]%
and   $\left\vert \operatorname{Re}\overline{z}w_{0}\right\vert \leq \left\vert
z\right\vert \left\vert w_{0}\right\vert \leq |z|^{2}$ and the integrand is
positive on  $z\in \Omega \backslash \Pi$, which implies $\Omega = \Pi$.
\end{proof}
\section{An inverse problem for Gabor localization}\label{Se:InvProbGab}
In this section we prove Theorem~\ref{Th:main1} and derive the complete
solution of the classical eigenvalue problem \eqref{eq:EVP} from the
assumption, that any single solution is a Hermite function.\\
In the sequel, we identify $(x,\xi )$
with $z=x+i\xi $ and we recall that $\pi (z)\varphi(t)  = \pi (x,\xi)\varphi(t) = \varphi(t-x)e^{2\pi i\xi t}$.

\subsection{Bargmann transform}
In the Gabor case, the choice of the
Gaussian function $\varphi (t)=2^{\frac{1}{4}}e^{-\pi t^{2}}$ allows the
translation of the time-frequency localization operator $H_{\chi _{\Omega
,\varphi }}$ to the complex analysis set-up via the \emph{Bargmann transform} 
$\mathcal{B}$. $\mathcal{B}f$ is defined for functions of a real variable as 
\begin{equation}
\mathcal{B}f(z)=\int_{\mathbb{R}}f(t)e^{2\pi tz-\pi t^{2}-\frac{\pi }{2}%
z^{2}}dt=e^{-i\pi x\xi +\pi \frac{\left\vert z\right\vert ^{2}}{2}}\mathcal{V%
}_{\varphi }f(x,-\xi )\text{.}  \label{Bargmann}
\end{equation}%
$\mathcal{B}$ maps $L^{2}(\mathbb{R})$ unitarily onto $\mathcal{F}^{2}(%
\mathbb{C})$, the Bargmann-Fock space of analytic functions with the inner
product obtained by choosing the measure $d\mu (z)=e^{-\pi |z|^{2}}dz.$
\subsection{The Hermite functions} The normalized  monomials $e_{n}=(\pi ^{n}/n!)\cdot z^{n}=%
\mathcal{B}h_{n}(z)=e^{-i\pi x\xi +\pi \frac{\left\vert z\right\vert ^{2}}{2}%
}\mathcal{V}_{\varphi }h_{n}(z)$ form an orthonormal basis for $\mathcal{F}%
^{2}(\mathbb{C})$.  Here,  $h_{n}(t) = c_n e^{\pi t^2} ( \frac{d}{
dt})^n(e^{-2\pi t^2 } )$ are the Hermite functions, which, by appropriate choice of  $c_n$, provide an orthonormal basis of  $L^2(\mathbb{R})$. As a direct consequence of the unitarity of $\mathcal{B}$ and $\mathcal{V}_\varphi$, the set $\{\mathcal{V}_\varphi h_n, n\in\mathbb{N}\}$ is 
orthogonal over all discs $D_{R}$. 
\subsection{Proof of Theorem~\protect\ref{Th:main1}}\label{Se:Proof}
We first deduce the equivalent formulation of the eigenvalue problem %
\eqref{eq:EVP} in the Bargmann domain.  Since the Bargmann transform is unitary, \eqref{eq:EVP}
is equivalent to 
\begin{equation*}
\int_{\Omega }\mathcal{V}_{\varphi }f(\overline{z})\mathcal{B}(\pi (z)\varphi )(w
)\,dz=\lambda \mathcal{B}f(w )
\end{equation*}%
Now, since $\mathcal{B}(\pi (z)\varphi )(w )=e^{-\pi ix\xi }e^{-\pi
|z|^{2}/2}e^{\pi w \overline{z}}$, we write the previous equation as 
\begin{equation*}
\int_{\Omega }\mathcal{V}_{\varphi }f(\overline{z})e^{-\pi ix\xi }e^{-\pi
|z|^{2}/2}e^{\pi w \overline{z}}dz=\lambda \mathcal{B}f(w)\text{.}
\end{equation*}%
Thus, by \eqref{Bargmann}, we have 
\begin{equation*}
\int_{\Omega }\mathcal{B}f(z)e^{\pi \overline{z}w-\pi \left\vert
z\right\vert ^{2}}dz=\lambda \mathcal{B}f(w)\text{.}
\end{equation*}%
By the unitarity of the Bargmann transform we conclude that the eigenvalue
problem \eqref{eq:EVP} on $L^2 (\mathbb{R})$ is equivalent to%
\begin{equation}\label{Eq:BargProb}
\int_{\Omega }F(z)e^{\pi \overline{z}w-\pi \left\vert z\right\vert
^{2}}dz=\lambda F(w)\text{, }
\end{equation}%
on $\mathcal{F}^2\left( \mathbb{C}\right) $. 
We may now expand the kernel $e^{\pi \overline{z}w}$ in its power series which
transforms the eigenvalue equation to 
\begin{equation}
\lambda F(w)=\sum_{n=0}^{\infty }\frac{\pi ^{n}}{n!}w^{n}\int_{\Omega }F(z)%
\overline{z}^{n}e^{-\pi \left\vert z\right\vert ^{2}}dz  \label{eq:Bargeig1}
\end{equation}%
Now we use the assumption that $z^{m}$ solves \eqref{eq:Bargeig1} for $\lambda =\lambda _{m}$,
in other words, that any of the solutions of \eqref{eq:EVP} is a Hermite
function. Setting $F(z)=z^{m}$ then gives%
\begin{equation*}
\lambda _{m}w^{m}=\sum_{n=0}^{\infty }\frac{\pi ^{n}}{n!}w^{n}\int_{\Omega
}z^{m}\overline{z}^{n}e^{-\pi \left\vert z\right\vert ^{2}}dz\text{.}
\end{equation*}%
By the identity theorem for analytic functions, this implies 
\begin{equation*}
\int_{\Omega }\overline{z}^{n}z^{m}e^{-\pi \left\vert z\right\vert
^{2}}dz=\lambda _{m}\frac{m!}{\pi ^{m}}\delta _{n,m}.
\end{equation*}%
In particular, setting $n=m+k$, 
\begin{equation}
\int_{\Omega }\left\vert z\right\vert ^{2m}\overline{z}^{k}e^{-\pi
\left\vert z\right\vert ^{2}}dz=\lambda \delta _{k,0},\text{ for all }k\geq 1%
\text{.}  \label{eq:exp_zero1}
\end{equation}%
Now  Proposition~\ref{Th:LocMon} can be applied and we conclude that $\Omega $ must be
the union of $ \frac{n}{2}$  annuli centered at $0$ for even $n$ and the union of a disk and $ \frac{n-1}{2}$ annuli centered at $0$ for odd $n$. In particular, for simply connected $\Omega$, we obtain a disk centered at zero.

\subsection{Consequences of  Theorem~\ref{Th:main1}}

\begin{Cor}Let $\Omega$ be simply connected. 
If the Gabor transform of one of the eigenfunctions of the localization
operator $H_{\Omega }$ has Gaussian growth, $O(e^{-\pi \left\vert
z\right\vert ^{2}})$, then $\Omega $ must be a disk. The same conclusion
holds, if some eigenfunction has Gaussian growth in both the time and the
frequency domains.
\end{Cor}

\begin{proof}
This is a consequence of the version of Hardy's uncertainty principle for the Gabor transform proved by Gr\"{o}chenig and
Zimmermann~\cite{grzi01}. They showed that, if \ $\mathcal{V}%
_{g}f(z)=O(e^{-\pi \left\vert z\right\vert ^{2}})$, then both $f$ and $g$
must be time-frequency shifts of a Gaussian function. Therefore, under the
hypotheses of the corollary, the Gaussian (which is the first Hermite
function) is an eigenfunction of the localization operator $H_{\Omega }$ and
by Theorem 1, $\Omega $ must be a disk. The second statement follows in a
similar fashion from the classical Hardy uncertainty principle~\cite{ha33}.
\end{proof}

The result of Theorem~\ref{Th:main1} immediately implies that the complete
solution of \eqref{eq:EVP} is given by the orthonormal basis of Hermite
functions.

\begin{Cor}
Assume that an orthonormal  basis of $L^2 (\mathbb{R})$ has doubly orthogonal Gabor transform with respect to the Gaussian window $\varphi$ and some domain $\Omega$: 
\begin{equation}\label{Eq:DON}
	\int_\Omega \mathcal{V}_\varphi \varphi_j (z)\overline{\mathcal{V}_\varphi  \varphi_{j'} (z)}dz = c_j \delta_{j,j'} .
\end{equation}
Let $\Omega$ be simply connected or of the form stated in Corollary~\ref{Cor_JordanCirc}(b). If, for any $j_{0}$, $\varphi _{j_{0}}=h_{j_{0}}$
is a Hermite function, then for every $j\geq 0,$%
\begin{equation*}
\varphi_j=h_j\text{.}
\end{equation*}
\end{Cor}
\begin{proof}
Note that an orthonormal basis  of $L^2 (\mathbb{R})$ satisfies \eqref{Eq:DON} if and only if it consists of eigenfunctions of the localization operator $H_\Omega$. 
Hence, we are in the situation of Theorem~\ref{Th:main1}, and $\Omega$ must be  disk centered at zero, the union of a disk and a finite number of annuli centered at zero or an  annulus centered at zero, respectively. This, in turn, implies that all eigenfunctions  are Hermite functions.
\end{proof}
\begin{remark}\label{Rem1}
Note the following consequence of Theorem~\ref{Th:main1}: if the localization
domain $\Omega$ is not a disk, then the function of optimal concentration
inside $\Omega$, in the sense of \eqref{Eq:TFcon},  cannot be a Gaussian window. On the other hand, it is well-known that Gaussian windows uniquely minimize the Heisenberg uncertainty relation. In this sense, disks seem to be the optimal domain for measuring time-frequency concentration. \\
Gaussian windows $\varphi$ (and also higher order Hermite functions) are a popular choice for the basic atom in the generation of Gabor frames, whose members are given as $\pi (\lambda) \varphi$, $\lambda\in\Lambda$ for some  discrete subgroup $\Lambda\subset\mathbb{R}^2$. A  popular choice of $\Lambda$ is 
$\Lambda = a\mathbb{Z}\times b\mathbb{Z}$, i.e. a rectangular lattice. In this case, the fundamental domain of $\Lambda$ in $\mathbb{R}^2$ is rectangular and thus, according to Theorem~\ref{Th:main1}, no Hermite function can be maximally concentrated inside the fundamental domain. This observation suggests that Gaussian or Hermitian windows are not an ideal choice for generating Gabor frames along a rectangular lattice. Although no proof for a precise statement exists, this observation has been made before and it is consistently confirmed in  numerical experiments.
In particular, in~\cite{best03} it is mentioned that Gabor frames generated by time-frequency shifting Gaussian pulses over a  hexagonal lattice have  better condition number than frames obtained via a corresponding rectangular lattice. It is well-known and was shown by Gauss in 1840 that hexagonal lattices provide the  densest packing of circles in the plane~\footnote{Weisstein, Eric W. "Circle Packing." From MathWorld--A Wolfram Web Resource. \url{http://mathworld.wolfram.com/CirclePacking.html} }. On the other hand, it is known that a Gabor frame with a Gaussian basic window is never tight~\cite{ga09}.\\
Motivated by our observations we formulate the following \emph{conjecture}: \\Given a fixed redundancy $red>1$, the condition number of a Gabor frame with Gaussian window $\varphi$ is optimal for a hexagonal lattice.

\end{remark}

\subsection{Remark (Due to Karlheinz Gr\"ochenig)}\label{RemKH}
Since Daubechies' results also extend to localization operators with symbols $\sigma$ other than indicator functions, stating that any radial symbol equally leads to localization operators diagonalized by the Hermite functions, one may ask the obvious question, whether a similar inverse statement to Theorem~\ref{Th:main1} can be expected for more general symbol classes than indicator functions. The following example shows that this is not true.

Let $H_\sigma$ be a time-frequency localization operator. Then, for every $N\in\mathbb{N}_0$ there exist non-negative, non-radial  symbols $\sigma$, such that $H_\sigma h_N = \lambda h_N$.\\
To construct $\sigma$, we proceed as in Section~\ref{Se:Proof} and consider the equivalent operator  on $\mathcal{F}^2\left( \mathbb{C}\right) $, i.e.
\begin{equation*}
T_\sigma F (w ) = \int_{\mathbb{C}}\sigma (z) F(z)e^{\pi \overline{z}w-\pi \left\vert z\right\vert
^{2}}dz\text{. }
\end{equation*} We then claim that $T_\sigma (z^N ) = \lambda z^N$ for some non-radial $\sigma$.  We fix $N\in\mathbb{N}_0$ and let
\[\sigma (z)  = \sigma (r e^{2\pi i  t } ) = \sigma_0 (r) + \sigma_1 (r ) \cdot (  e^{2\pi i(N+1)  t }+e^{-2\pi i (N+1) t }),\]
where $\sigma_0 (r) \geq 2 | \sigma_1 (r)|$ $\forall r\geq 0$ and $\int_0^\infty \sigma_1 (r) r^{3N+1} e^{-\pi r^2} r\,dr = 0$. Observe that $\sigma_1$ can be chosen to be bounded, compactly supported and real-valued. Then we have $\sigma (z)  \geq \sigma_0 (r) -2| \sigma_1 (r )|\geq 0$. Since $\sigma_0$ is radial, we have 
$T_{\sigma_0} (z^N ) = \lambda_N z^N$ with $\lambda_N>0$. Therefore, it is enough to show that $T_{\sigma_1} (z^N ) = 0$. However, this is easy to see by considering 
$\sigma_+ = \sigma_1 (r )  e^{2\pi (N+1)i  t }$ and $\sigma_- = \sigma_1 (r ) e^{-2\pi i (N+1) t }$ separately and noting that, since $T_{\sigma_{\pm}}$ is entire, the task is reduced to showing that $\frac{d^l}{dw^l}  T_{\sigma_{\pm}}(z^N) |_{w = 0} = 0$ for all $l\in\mathbb{N}_0$. A straightforward calculation shows that, setting $F(z) = z^N$ and writing 
$( T_{\sigma_{\pm}}(f))^{(l)} = \frac{d^l}{dw^l}  T_{\sigma_{\pm}}(f)$,
 we have 
 \[( T_{\sigma_{\pm}}(F))^{(l)} (0 )  = \int_{\mathbb{C}}\sigma_{\pm} (z) z^N(\pi\overline{z})^l e^{-\pi |z|^2} dz.\]
We finally substitute polar coordinates $z = re^{2\pi i t}$ to obtain, for $\sigma_+$:
 \[( T_{\sigma_{\pm}}(F))^{(l)} (0 )  = \pi^l \int_0^\infty \sigma_1 (z) r^{N+l} e^{-\pi r^2} r\,dr\int_0^1 e^{2\pi i (2N+1 -l)t} dt.\] The integral over $t$ is zero for $l\neq 2N+1$  by orthogonality of the Fourier basis and the integral over $r$ is zero for $l =  2N+1$ by assumption. The argument for  $\sigma_-$ is similar.

\section{An inverse problem for wavelet localization}

\label{SE:WavLoc}

By replacing \textquotedblleft Gabor transform\textquotedblright\ by
\textquotedblleft wavelet transform\textquotedblright\ in the formulation of
the inverse problem for time-frequency localization, we may define a
completely analogous inverse problem for wavelet localization. The corresponding direct
problem has been treated by Daubechies and Paul in~\cite{dapa88} and by Seip
in~\cite{se91}. This section is related to the previous one in the same way
as the direct problem studied in~\cite{dapa88} is related to the problem
studied in~\cite{da88}. It is quite remarkable that, after appropriate
reformulation of the eigenvalue problem, we can apply Proposition~\ref{Th:LocMon}
to wavelet localization operators. Since our arguments depend on the
connection to complex variables, it is essential to consider the Hardy space
of the upper-half plane as the domain of the wavelet transform. Then, we
choose a certain analyzing wavelet which plays the role of the Gaussian and
the localization problem can be reformulated in certain weighted Bergman
spaces. This basic strategy follows the lines which lead to the
Bargmann-Fock space formulation in the Gabor case.

One relevant difference between the wavelet and the Gabor case stems from
the hyperbolic geometry of the upper-half plane. Since the set-up of
Proposition~\ref{Prop1} is not visible in the spaces defined on the
half-plane, we will translate the problem to a conformally equivalent
hyperbolic region: the unit disc. There, the problem finds a natural
formulation and Proposition~\ref{Th:LocMon} applies. This point of view is
suggested by Seip's approach in~\cite{se91}. In short, while in the Gabor case the Bargmann
transform maps $L^2(\mathbb{R})$ to the Bargmann-Fock space, where the
monomials are orthogonal,%
\begin{equation*}
\mathcal{B}:L^{2}(\mathbb{R})\rightarrow \mathcal{F}^{2}(\mathbb{C})
\end{equation*}%
we now need to further transform the images of the so called Bergman
transform ($Ber_{\alpha }$) to a space defined in the unit disc. This
transformation is given by a Cayley transform $T_{\alpha }$, as defined in
Section~\ref{SS:BergSp}: 
\begin{equation}
H^{2}(\mathbb{C}^{+})\overset{Ber_{\alpha }}{\rightarrow }A_{\alpha }(%
\mathbb{C}^{+})\overset{T_{\alpha }}{\rightarrow }A_{\alpha }(\mathbb{D})%
\text{.}  \label{isomo_wavelet}
\end{equation}%
The role of the Hermite functions is taken over by special functions, whose
Fourier transforms are the Laguerre functions. This is possible, since the Laguerre functions constitute an orthogonal basis for $L^2 (0,\infty)$ and the Fourier transform provides a unitary isomorphism  $H^{2}(\mathbb{C}^{+})\rightarrow L^2 (0,\infty)$.

\subsection{The wavelet transform}

Since analyticity will play a fundamental role, in this section we restrict
ourselves to functions in a subspace of $L^2(\mathbb{R})$, namely to $f\in H^{2}(\mathbb{C}^{+})$, the Hardy space in the
upper half plane. $H^{2}(\mathbb{C}^{+})$ is constituted by analytic functions $f$ such that 
\begin{equation*}
\text{ }\sup_{0<s<\infty }\int_{-\infty }^{\infty }\left\vert
f(x+is)\right\vert ^{2}dx<\infty \text{.}
\end{equation*}%
The functions in the space $H^{2}(\mathbb{C}^{+})$ may be considered as being  of
\textquotedblleft positive frequency\textquotedblright\ since a well known
Paley-Wiener theorem says that $\mathcal{F}(H^{2}(\mathbb{C}%
^{+}))=L^{2}(0,\infty )$. For this reason it is common to study $H^{2}(%
\mathbb{C}^{+})$ on the \textquotedblleft frequency side\textquotedblright ,
where many calculations become easier. For convenience we will use a different normalization of the Fourier transform in this section, namely $(%
\mathcal{F}f)(\xi )=\left( 2\pi \right) ^{-\frac{1}{2}}\int_{-\infty
}^{\infty }e^{-i\xi t}f(t)dt$. Now consider $%
\mathbb{C}
^{+}=\{z\in 
\mathbb{C}
:\operatorname{Re}(z)>0\}.$ For every $x\in \mathbb{R}$ and $s\in \mathbb{R}^{+}$,
let $z=x+is\in 
\mathbb{C}
^{+}$ and define%
\begin{equation*}
\pi _{z}g(t)=s^{-\frac{1}{2}}g(s^{-1}(t-x)).
\end{equation*}%
Fix a function $g\neq 0$ such that 
\begin{equation*}
0<\left\Vert \mathcal{F}g\right\Vert _{L^{2}(%
\mathbb{R}
^{+},t^{-1})}^{2}=C_{g} <\infty.
\end{equation*}%
Such functions are called \emph{admissible }and the constant $C_{g}$ is the 
\emph{admissibility constant}. Then \emph{the continuous wavelet transform}
of a function $f$ with respect to a wavelet\ $g$ is defined, for every$z=x+is\in 
\mathbb{C}$ as 
\begin{equation}
W_{g}f(z)=\left\langle f,\pi _{z}g\right\rangle _{H^{2}\left( \mathbb{C}%
^{+}\right) }.  \label{wavelet}
\end{equation}%
Let $d\mu ^{+}(z)$ denote the standard normalized area measure in $\mathbb{C}^{+}$%
. The orthogonal relations for the wavelet transform%
\begin{equation}
\int_{\mathbb{C}^{+}}W_{g_{1}}f_{1}(x,s)W_{g_{2}}f_{2}(x,s)s^{-2}d\mu
^{+}(z)=\left\langle \mathcal{F}g_{1},\mathcal{F}g_{2}\right\rangle _{L^{2}(%
\mathbb{R}
^{+},t^{-1})}\left\langle f_{1},f_{2}\right\rangle _{H^{2}\left( \mathbb{C}%
^{+}\right) }\text{,}  \label{ortogonalityrelations}
\end{equation}%
are valid for all $f_{1},f_{2}\in H^{2}\left( \mathbb{C}^{+}\right) $ and $%
g_{1},g_{2}\in H^{2}\left( \mathbb{C}^{+}\right) $ admissible. As a result,
the continuous wavelet transform provides an isometric inclusion%
\begin{equation*}
W_{g}:H^{2}\left( \mathbb{C}^{+}\right) \rightarrow L^{2}(\mathbb{C}^{+}%
\mathbf{,}s^{-2}dxds)\text{,}
\end{equation*}%
which is an isometry for $C_{g}=1$.

\subsection{Bergman spaces}

\label{SS:BergSp}

{Let }$\alpha >-1$ and $d\mu ^{+}(z)$ denote  the standard normalized area measure  in $\mathbb{C}%
^{+}$. The \emph{Bergman space in the upper half plane}, $%
A_{\alpha }(\mathbb{C}^{+})$, is constituted by the analytic functions in $%
\mathbb{C}^{+}$ such that 
\begin{equation}
\int_{\mathbb{C}^{+}}\left\vert f(z)\right\vert ^{2}s^{\alpha }d\mu
^{+}(z)<\infty \text{.}  \label{normupperBergman}
\end{equation}%
Now consider $\mathbb{D=\{}z\in \mathbb{C}:\left\vert z\right\vert <1%
\mathbb{\}}$ with normalized area measure $dA(w)$.\ The \emph{Bergman space in the unit disc}, denoted by $%
A_{\alpha }(\mathbb{D})$, is constituted by the analytic functions in $%
\mathbb{D}$ such that 
\begin{equation}
\int_{\mathbb{D}}\left\vert f(w)\right\vert ^{2}(1-\left\vert w\right\vert
)^{\alpha }dA(w)<\infty \text{.}  \label{normbergman}
\end{equation}%
 The map $%
T_{\alpha }:A_{\alpha }(\mathbb{C}^{+})\rightarrow A_{\alpha }(\mathbb{D})$, defined as 
\begin{equation*}
(T_{\alpha }f)(w)=\frac{2^{ \frac{\alpha}{2}+1}}{( 1-w)^{\alpha +2}}f\left( 
\frac{w+1}{i(w-1)}\right) \text{,}
\end{equation*}%
provides a unitary isomorphism between the two spaces. The reproducing kernel of $%
A_{\alpha }(\mathbb{C}^{+})$ is%
\begin{equation*}
\mathcal{K}_{\mathbb{C}^{+}}^{\alpha }(z,w)=\left( \frac{1}{w-\overline{z}}%
\right) ^{\alpha +2}\text{.}
\end{equation*}%
\ Now observe that, letting $T_{\alpha }$ act on the reproducing kernel of $%
A_{\alpha }(\mathbb{C}^{+})$, first as a function of $w$ and then as a
function of $\overline{z}$, we are led to the reproducing kernel of $%
A_{\alpha }(\mathbb{D})$,  
\begin{equation}
\mathcal{K}_{\mathbb{D}}^{\alpha }(z,w)=\frac{1}{(1-w\overline{z})^{\alpha +2
}}\text{.}  \label{repBergman}
\end{equation}

\subsection{The Bergman transform}
Let $\Gamma $\ denote  the Gamma function and set 
\begin{equation*}
c_{\alpha }^{2}=\int_{0}^{\infty }t^{2\alpha -1}e^{-2t}dt=2^{2\alpha
-1}\Gamma (2\alpha )\text{,}
\end{equation*}
We can relate the wavelet transform to Bergman spaces of analytic functions by choosing
 the window $\psi _{\alpha }$ as
\begin{equation}
\mathcal{F}\psi _{\alpha }(t)=\frac{1}{c_{\alpha }}\mathbf{1}_{\left[
0,\infty \right] }t^{\alpha }e^{-t}\text{.}  \label{fg0}
\end{equation} 
 The choice of  $c_{\alpha }$ implies $%
C_{\psi _{\alpha }}=1$ and the corresponding wavelet transform is isometric.
The \emph{Bergman transform }of order\emph{\ }$\alpha $ is the unitary map$\
Ber_{\alpha }$ $:H(\mathbb{C}^{+})\rightarrow A_{\alpha }(\mathbb{C}^{+})$
given by 
\begin{equation}
Ber_{\alpha }\text{ }f(z)=s^{-\frac{\alpha }{2}-1}W_{\overline{\psi} _{\frac{%
\alpha +1}{2}}}f(-x,s)=c_{\alpha }\int_{0}^{\infty }t^{\frac{\alpha +1}{2}}(%
\mathcal{F}f)(t)e^{izt}dt\text{.}  \label{Ber}
\end{equation}

\subsection{The Laguerre and other related systems of functions}

We define the Laguerre functions 
\begin{equation*}
\emph{l}_{n}^{\alpha }(x)=\mathbf{1}_{\left[ 0,\infty \right]
}(x)e^{-x/2}x^{\alpha /2}L_{n}^{\alpha }(x)\text{.}
\end{equation*}
in terms of the Laguerre polynomials 
\begin{equation}
L_{n}^{\alpha }(x)=\frac{e^{x}x^{-\alpha }}{n!}\frac{d^{n}}{dx^{n}}\left[
e^{-x}x^{\alpha +n}\right] =\sum_{k=0}^{n}(-1)^{k}\binom{n+\alpha }{n-k}%
\frac{x^{k}}{k!}.  \label{Rodr}
\end{equation}%
\ By repeated integration by parts, one sees that the polynomials $%
L_{n}^{\alpha }(x)$ are orthogonal on $(0,\infty )$ with respect to the
weight function $e^{-x}x^{\alpha }$ .\ Thus, for $\alpha \geq 0$, the Laguerre
functions $\emph{l}_{n}^{\alpha }$ constitute an orthogonal basis for the space $L^{2}(0,\infty )$.
We will use a related system of functions $\psi _{n}^{\alpha }$ defined as 
\begin{equation*}
\left( \mathcal{F}\psi _{n}^{\alpha }\right) (t)=\left( \frac{(-1)^{n}n!}{%
2^{2\alpha +2n+1}\Gamma (n+2+\alpha )\Gamma (2+\alpha )}\right) ^{\frac{1}{2}%
}l_{n}^{\alpha +1}(2t)\text{.}
\end{equation*}
Now consider the monomials 
\begin{equation*}
e_{n}^{\alpha }(w)=\left( \frac{\Gamma (n+2+\alpha )}{n!\Gamma (2+\alpha )}%
\right) ^{\frac{1}{2}}w^{n}.
\end{equation*}%
We can apply Proposition 1 with $\mu (\left\vert z\right\vert
)=(1-\left\vert w\right\vert ^{2})^{\alpha }$. We  conclude that $%
\{e_{n}^{\alpha }\}_{n=0}^{\infty }$ \ forms an orthonormal basis for $%
\mathcal{A}_{\alpha }(\mathbb{D})$ and that they are orthogonal on every
disk $D_{r}\subset \mathbb{D}$: for every $r>0$,%
\begin{equation}
\int_{D_{r}}e_{n}^{\alpha }(w)\overline{e_{m}^{\alpha }(w)}(1-\left\vert
w\right\vert ^{2})^{\alpha }dA(w)=C(r,m)\delta _{nm}\text{.}  \label{ortdisk}
\end{equation}%
The normalization constant $C(r,m)$ depends on $r$ and $m$ and satisfies $%
\lim_{r\rightarrow 1^{-}}C(r,m)=1$. Now, the functions $\Psi _{n}^{\alpha }$
, for every $n\geq 0$ and $\alpha >-1$,%
\begin{equation*}
\Psi _{n}^{\alpha }(z)=\frac{1}{4^{\alpha +\frac{1}{2}}}\left( \frac{\Gamma
(n+2+\alpha )}{n!\Gamma (2+\alpha )}\right) ^{\frac{1}{2}}\left( \frac{z-i}{%
z+i}\right) ^{n}\left( \frac{1}{z+i}\right) ^{\alpha +2}\text{, \ \ \ \ \ }%
z\in 
\mathbb{C}
^{+}\text{,}
\end{equation*}%
are conveniently choosen such that%
\begin{equation}
(T_{\alpha }\Psi _{n}^{2\alpha })(w)=e_{n}^{\alpha }(w).  \label{translation}
\end{equation}%
Thus, a change of variables $w=\frac{z-i}{z+i}$ in (\ref{ortdisk}) leads to%
\begin{equation}
\int_{\varrho (z,i)<r}\Psi _{n}^{\alpha }(z)\overline{\Psi _{n}^{\alpha }(z)}%
s^{\alpha }d\mu ^{+}(z)=C(r,m)\delta _{nm}\text{,}  \label{double_ort}
\end{equation}%
where $\varrho (z_{1},z_{2})=\left\vert \frac{z_{1}-z_{2}}{z_{1}-\overline{%
z_{2}}}\right\vert $ is the pseudohyperbolic metric on $\mathbb{C}^{+}$.
Moreover, the unitarity of the operator $T_{\alpha }$ translates the basis
property of  $\{e_{n}^{\alpha }\}_{n=0}^{\infty }$ \ in $\mathcal{A}_{\alpha
}(\mathbb{D})$ to $\mathcal{A}_{\alpha }(\mathbb{C}^{+})$. In other words, (%
\ref{translation}) shows that $\{\Psi _{n}^{\alpha }(z)\}_{n=0}^{\infty }$
is an orthogonal basis of $\mathcal{A}_{\alpha }(\mathbb{C}^{+})$. \ Finally
we observe that (\ref{Ber}) together with the special function formula 
\begin{equation}
\int_{0}^{\infty }x^{\alpha }L_{n}^{\alpha }(x)e^{-xs}dx=\frac{\Gamma
(\alpha +n+1)}{n!}s^{-\alpha -n-1}(s-1)^{n}  \label{formlaguerre}
\end{equation}%
gives 
\begin{equation*}
Ber_{\alpha }\text{ }\psi _{n}^{\alpha }=\Psi _{n}^{\alpha }\text{.}
\end{equation*}%
For an intuitive grasp of this section, keep in mind that to the composition
of transforms (\ref{isomo_wavelet}) one associates the transformations of
the basis functions:%
\begin{equation*}
\psi _{n}^{\alpha }\in H(\mathbb{C}^{+})\overset{Ber_{\alpha }}{\rightarrow }%
\Psi _{n}^{\alpha }\in A_{\alpha }(\mathbb{C}^{+})\overset{T_{\alpha }}{%
\rightarrow }e_{n}^{\alpha }\in A_{\alpha }(\mathbb{D}).
\end{equation*}

\subsection{The inverse problem}

We now consider the wavelet localization operator $P_{\Delta ,g}$ defined as%
\begin{equation*}
P_{\Delta ,\alpha }f=\int_{\Delta }W_{\overline{\psi} _{\frac{\alpha +1}{2}}}%
f(z)\pi _{z}\overline{\psi} _{\frac{\alpha +1}{2}}(t)d\mu ^{+}(z)
\end{equation*}
and set up the corresponding eigenvalue problem%
\begin{equation}  \label{Eq:WavEIG}
P_{\Delta ,\alpha }f=\lambda f
\end{equation}

\begin{theorem}
\label{Th:main2} If one of the eigenfunctions of the localization operator $%
P_{\Delta ,\alpha }$ belongs to the family $\{\psi _{n}^{\alpha }\}$, then $%
\Delta $ must be a pseudohyperbolic disc centered at $i$.
\end{theorem}

\begin{proof}
We first rewrite the eigenvalue problem \eqref{Eq:WavEIG}. A simple change
of variables on the \textquotedblleft Fourier\textquotedblright\ side of the
wavelet representation gives 
\begin{equation*}
Ber_{\alpha }(\pi _{z}\psi _{\frac{\alpha +1}{2}})(w)=m_{\alpha }s^{\frac{%
\alpha +2}{2}}\left( \frac{1}{z-\overline{w}}\right) ^{\alpha +2}=s^{\frac{%
\alpha +2}{2}}\mathcal{K}_{\mathbb{C}^{+}}^{\alpha }(z,w),
\end{equation*}%
where $m_{\alpha }=\frac{\alpha +1}{2^{\alpha }}$. Now apply the Bergman
transform and use (\ref{Ber}) to rewrite \eqref{Eq:WavEIG} as 
\begin{equation*}
\int_{\Delta }Ber_{\alpha }\text{ }f(z)\mathcal{K}_{\mathbb{C}^{+}}^{\alpha
}(z,w)s^{\alpha }d\mu ^{+}(z)=\lambda Ber_{\alpha }\text{ }f(w)\text{.}
\end{equation*}%
By the unitarity $Ber_{\alpha }$ $:H(\mathbb{C}^{+})\rightarrow A_{\alpha }(%
\mathbb{C}^{+})$ we conclude that our eigenvalue problem is equivalent to%
\begin{equation*}
\int_{\Delta }F(z)\mathcal{K}^\alpha_{\mathbb{C}^{+}}(z,w)s^{\alpha }d\mu
^{+}(z)=\lambda F(w)\text{, }
\end{equation*}%
with $F\in A_{\alpha }(\mathbb{C}^{+})$. Making the change of variables%
\begin{equation*}
z_{\mathbb{D}}=\frac{z-i}{z+i},w_{\mathbb{D}}=\frac{w-i}{w+i}\text{,}
\end{equation*}%
we move the eigenvalue problem to the unit disc%
\begin{equation*}
\int_{\Omega =T_{\alpha }\Delta }(T_{\alpha }F)(z_{\mathbb{D}})\frac{%
(1-\left\vert z_{\mathbb{D}}\right\vert )^{\alpha }}{(1-w_{\mathbb{D}}%
\overline{z_{\mathbb{D}}})^{2+\alpha }}dA_{\mathbb{D}}(z_{\mathbb{D}%
})=\lambda (T_{\alpha }F)(w_{\mathbb{D}})\text{,}
\end{equation*}%
where $T_{\alpha }F(w_{\mathbb{D}})\in A_{\alpha }(\mathbb{D})$. We simplify
the notation writing $z_{\mathbb{D}}=z,w_{\mathbb{D}}=z$. Now, using the
uniformly convergent expansion of the reproducing kernel, 
\begin{equation*}
\frac{1}{(1-w\overline{z})^{2+\alpha }}=\sum_{n=0}^{\infty }e_{n}^{\alpha
}(w)\overline{e_{n}^{\alpha }(z)},
\end{equation*}%
we can transform the eigenvalue equation into 
\begin{equation}
\lambda (T_{\alpha }F)(w)=\sum_{n=0}^{\infty }\frac{\Gamma (n+2+\alpha )}{%
n!\Gamma (2+\alpha )}w^{n}\int_{\Omega }(T_{\alpha }F)(z)\overline{z^{n}}%
(1-\left\vert z\right\vert )^{\alpha }dA_{\mathbb{D}}(z)  \label{eigenwav_2}
\end{equation}%
If one of the eigenfunctions of the localization operator $P_{\Delta ,\alpha
}$ belongs to the family $\{\psi _{n}^{\alpha }\}$, then 
\begin{equation*}
T_{\alpha }(Ber_{\alpha }\psi _{n}^{\alpha })(z)=T_{\alpha }(\Psi
_{n}^{\alpha })(z)=e_{n}^{\alpha }(z)
\end{equation*}%
solves (\ref{eigenwav_2}) for $\lambda =\lambda _{n}$. Setting $(T_{\alpha
}F)(z)=e_{m}^{\alpha }(z)$ gives 
\begin{equation*}
\lambda _{m}w^{m}=\sum_{n=0}^{\infty }\frac{\Gamma (n+2+\alpha )}{n!\Gamma
(2+\alpha )}w^{n}\int_{\Omega }z^{m}\overline{z}^{n}(1-\left\vert
z\right\vert )^{\alpha }dA_{\mathbb{D}}(z)\text{.}
\end{equation*}%
Comparison of coefficients yields 
\begin{equation*}
\int_{\Omega }\overline{z}^{n}z^{m}e^{-\pi \left\vert z\right\vert
^{2}}(1-\left\vert z\right\vert )^{\alpha }dA_{\mathbb{D}}(z)=\lambda _{m}%
\frac{n!\Gamma (2+\alpha )}{\Gamma (n+2+\alpha )}\delta _{n,m}
\end{equation*}%
and further, with $n=m+k$,  the condition of Proposition~\ref{Th:LocMon}: 
\begin{equation}
\int_{\Omega }\left\vert z\right\vert ^{2m}\overline{z}^{k}(1-\left\vert
z\right\vert )^{\alpha }dA_{\mathbb{D}}(z)=\lambda \delta _{k,0},\text{ for
all }k\geq 1\text{.}
\end{equation}%
Hence Proposition~\ref{Th:LocMon} can be applied and $\Omega $ must be a disk centered
at zero. Now we can go back to the upper half plane by the change of
variables 
\begin{equation*}
u=i\frac{z+1}{1-z}\in \mathbb{C}^{+}\text{,}
\end{equation*}%
which maps $0\mathbb{\in D}$ to $i\in \mathbb{C}^{+}$ and leaves the
pseudohyperbolic metric invariant. Finally, notice that the condition $%
\left\vert z\right\vert <r$, can be written in terms of the pseudohyperbolic
metric of the disc as $\varrho _{\mathbb{D}}(z,0)<r$. Hence, the disc $%
\Omega $ centered at zero is mapped to the pseudohyperbolic disc $\Delta
=\{\varrho _{\mathbb{C}^{+}}(u,i)<r\}$ centered at $u=i$.
\end{proof}
\begin{remark}
We can draw a conclusion similar to the one in Remark~\ref{Rem1} after Theorem~\ref{Th:main1}.
Indeed, it follows from 
Theorem~\ref{Th:main2} 
that, if the localization domain
 $\Delta $ is not a
pseudohyperbolic disc, then the function $f$ providing  optimal concentration in the sense of maximizing
\begin{equation}\label{Eq:TScon}
	\mathcal{C}_\Delta (f) = \frac{\int_\Delta |W_{\psi_\alpha} f (z)|^2 dz}{\|f\|_{H^{2}(%
\mathbb{C}^{+})}^2}\, 
\end{equation}
cannot be  the function  $\psi_\alpha$ in \eqref{fg0} - the so called Cauchy wavelet. On the
other hand it is known that the functions  $\psi_\alpha$ minimize the affine uncertainty
principle as first mentioned in \cite{pa84}. In this sense, pseudohyperbolic discs seem to be optimal domains
to measure wavelet localization.
\end{remark}	
\section*{Acknowledgments}
We would like to thank Saptarshi Das for his precious advice on UWB technology.
 \end{sloppypar}

\end{document}